\documentclass[letterpaper,11pt]{article}

\usepackage{amsmath,amsthm,amssymb}

\usepackage{graphicx, hyperref,  color}
\usepackage{fullpage}

\usepackage{algpseudocode, algorithm}
\usepackage{proba}

\newcommand{\Exp}{\mathbf{E}}
\newcommand{\Prob}{\mathbf{P}}
\def\<#1,#2>{\langle #1,#2\rangle}

\newcommand{\ivar}{ y} % variable in inner loop
\newcommand{\ovar}{ x} % variable in outer loop
\newcommand{\iidx}{ t} % index in inner loop
\newcommand{\oidx}{ k} % inder in outer loop
\newcommand{\rsample}{ i} % index of sampled row
\newcommand{\csample}{ j} % index of sampled column
\newcommand{\weightinG}{q} 
\newcommand{\weightofnorm}{v}

\newcommand{\hL}{ \hat{L}}
\newcommand{\stepsize}{h}

\newcommand{\eqdef}{:=}

\newtheorem{theorem}{Theorem}
\newtheorem{lemma}[theorem]{Lemma}

\newtheorem{corollary}[theorem]{Corollary}

\title{Semi-Stochastic Coordinate Descent\thanks{JK acknowledges support from Google through the Google European Doctoral Fellowship in Optimization Algorithms. ZQ and PR would like to  acknowledge support from the EPSRC  Grant EP/K02325X/1,
{\em Accelerated Coordinate Descent Methods for Big Data Optimization}. }}

\date{October 16, 2014 \\(full\thanks{A short version of this paper (5 pages; including the main result but without proof) was posted  on arXiv on October 16, 2014 \cite{s2cd-nips}. The paper was accepted  for presentation at the 2014 NIPS Optimization for Machine Learning workshop in a peer reviewed process. The accepted papers are listed on the website of the workshop, but are not published in any proceedings volume.} version: December 18, 2014)}

\author{
Jakub Kone\v{c}n\'{y} \qquad 
Zheng Qu \qquad
Peter Richt\'{a}rik  \\\\
{\em School of Mathematics}\\
{\em University of Edinburgh}\\
{\em United Kingdom}
}

\begin{document}

\maketitle

\begin{abstract}
We propose a novel  stochastic gradient method---semi-stochastic coordinate descent (S2CD)---for the problem of minimizing a strongly convex function  represented as the average of a large number of smooth convex functions:  $f(x)=\tfrac{1}{n}\sum_i f_i(x)$. Our method first performs a deterministic step (computation of the gradient of $f$  at the starting point), followed by a large number of stochastic steps. The process is repeated a few times, with the last stochastic iterate becoming the new starting point where the deterministic step is taken. The novelty of our method is in how the stochastic steps are performed. In each such step, we pick a random function $f_i$ and a random coordinate $j$---both using nonuniform distributions---and update a single coordinate of the decision vector only, based on the computation of  the $j^{th}$ partial derivative of $f_i$ at two different points. Each random step of the method constitutes an unbiased estimate of the gradient of $f$ and moreover, the squared norm of the steps goes to zero in expectation, meaning that  the stochastic estimate of the gradient progressively improves.   The complexity of the method is the sum of two terms: $O(n\log(1/\epsilon))$ evaluations of gradients  $\nabla f_i$ and $O(\hat{\kappa}\log(1/\epsilon))$ evaluations of partial derivatives $\nabla_j f_i$, where $\hat{\kappa}$ is a novel condition number. 
\end{abstract}

\newpage
\section{Introduction}

In this paper we study the problem of  unconstrained minimization of a strongly convex function represented as the average of a large number of smooth convex functions:
\begin{equation} 
\label{eq:main} 
%\textstyle 
\min_{x \in \R^d} f(x) \equiv \frac{1}{n} \sum_{i=1}^n f_i(x).
\end{equation}

Many computational problems in various disciplines are of this form. In machine learning,  $f_i(x)$ represents the loss/risk of classifier $x\in \R^d$ on data sample $i$, $f$ represents the empirical risk (=average loss), and the goal is to find a predictor minimizing  $f$. Often, an L2-regularizer of the form $\mu \|x\|^2$, for $\mu>0$, is added to the loss, making it strongly convex and hence easier to minimize.

%This is typically the case the $f_i$ represent a loss function $\ell$, dependent on the point $x$ and particular data point, for example $f_i(x) = \ell(x, a_i, y_i) + \frac{\lambda}{2} \| x \|^2$, where $a_i$ would be the $i^{th}$ training examples, $y_i$ its label and $\lambda$ a regularization parameter. 

\paragraph{Assumptions.}  We assume that the functions  $f_i : \R^d \rightarrow \R$ are  differentiable and convex function, with Lipschitz continuous partial derivatives. Formally, we assume that for each  $i \in [n]\eqdef \{1,2,\dots,n\}$ and  $j \in [d]\eqdef \{1,2,\dots,d\}$ there exists $L_{ij}\geq 0$ such that for all $x\in \R^d$ and $h\in \R$, 
\begin{equation}\label{eq:sjs7shd}
f_\rsample(x + he_\csample) \leq f_\rsample(x) + \left\langle \nabla f_\rsample(x), h e_\csample \right\rangle + \frac{L_{\rsample \csample}}{2} h ^2,
\end{equation}
where $e_\csample$ is the $j^{th}$ standard basis vector in $\R^d$, $\nabla f(x)\in\R^d$ the gradient of $f$ at point $x$ and $\<\cdot, \cdot >$ is the standard inner product.  This assumption was recently used in the analysis the   accelerated coordinate descent method APPROX \cite{approx}. We further assume that  $f$ is $\mu$-strongly convex. That is, we assume that there exists $\mu>0$ such that for all $x,y\in \R^d$,
\begin{equation}
\label{SVRGstrcvx}
f(y) \geq f(x) + \langle \nabla f(x), y - x \rangle + \frac{\mu}{2} \| y - x \|^2.
\end{equation}

\paragraph{Context.}  Batch methods such as gradient descent (GD) enjoy a fast  (linear) convergence rate: to achieve $\epsilon$-accuracy, GD  needs $\mathcal{O}(\kappa\log(1 / \epsilon))$ iterations, where $\kappa$ is a condition number. The drawback of GD is  that in each iteration one needs to compute the gradient of $f$, which requires a pass through the entire dataset. This is prohibitive to do many times if $n$ is very large.

Stochastic gradient descent (SGD)  in each iteration computes the gradient of  a single randomly chosen function $f_i$ only---this constitutes  an unbiased (but noisy) estimate of the gradient of $f$---and makes a step in that direction \cite{robbinsmonro,sgdnemirovski, sgdzhang}. The rate of convergence of SGD is slower, $\mathcal{O}(1 / \epsilon)$,  but the cost of each iteration is independent of $n$. Variants with nonuniform selection probabilities were considered in \cite{zhaozhang-importance_sampling}, a mini-batch variant (for SVMs with hinge loss) was analyzed  in \cite{takac-minibatch}.

Recently, there has been progress in designing algorithms that achieve the fast $O(\log(1/\epsilon))$ rate without the need to scan the entire dataset in each iteration. The first class of methods to have achieved this are stochastic/randomized coordinate descent methods.

When applied to \eqref{eq:main}, coordinate descent methods (CD) \cite{nesterov, rtsimple}  can, like SGD,  be seen as an attempt to keep the benefits of GD (fast linear convergence) while reducing the complexity of each iteration. A CD method only computes a single partial derivative  $\nabla_j f(x)$ at each iteration  and updates a single coordinate of vector $x$ only. When  chosen uniformly at random, partial derivative is also an unbiased estimate of the gradient. However, unlike the SGD estimate, its variance is low. Indeed, as one approaches the optimum, partial derivatives decrease to zero. While CD methods are able to obtain linear convergence, they typically need $O((d/\mu)\log(1/\epsilon))$ iterations when applied to \eqref{eq:main} directly  \footnote{The complexity can be improved to $O(\tfrac{d\beta}{\tau \mu}\log(1/\epsilon))$ in the case when $\tau$ coordinates are updated in each iteration, where $\beta \in [1,\tau]$ is a problem-dependent constant \cite{pcdm}. This has been further studied  for nonsmooth problems via smoothing \cite{spcdm}, for arbitrary nonuniform distributions governing the selection of coordinates \cite{nsync,quartz} and in the distributed setting \cite{hydra, hydra2,quartz}. Also, efficient accelerated variants with $O(1/\sqrt{\epsilon})$ rate were developed \cite{approx, hydra2}, capable of solving problems with 50 billion variables.}. CD method typically significantly outperform GD, especially  on sparse problems with a very large number variables/coordinates \cite{nesterov, rtsimple}. 

An alternative to applying CD to \eqref{eq:main} is to apply it to the dual problem. This is possible under certain additional structural assumptions on the functions $f_i$. This is the strategy employed by stochastic dual coordinate ascent (SDCA) \cite{sdca,quartz}, whose rate is \[O((n+\kappa)\log(1/\epsilon)).\] The condition number $\kappa$ here is different (and larger) than the condition number appearing in the rate of GD. Despite this, this is a vast improvement on the rates achieved by both GD and SGD, and the method indeed  typically performs much better in practice. Accelerated \cite{sdca-accel} and mini-batch \cite{takac-minibatch} variants of SDCA have also been  proposed. We refer the reader to QUARTZ \cite{quartz} for a general analysis involving the update of a random subset of dual coordinates, following an arbitrary distribution.

Recently, there has been progress in designing primal methods which match the fast rate of SDCA.  Stochastic average gradient (SAG)  \cite{sag}, and more recently SAGA \cite{saga}, move in a direction composed of old stochastic gradients. The semi-stochastic gradient descent (S2GD) \cite{s2gd, mS2GD} and stochastic variance reduced gradient (SVRG) \cite{svrg, proxsvrg} methods employ a different strategy: one first computes the gradient of $f$, followed by $O(\kappa)$ steps where only stochastic gradients are computed. These are used to estimate the change of the gradient, and it is this direction which combines the old gradient and the new stochastic gradient information which used in the update.

\paragraph{Main result.} In this work we develop a new  method---semi-stochastic coordinate descent (S2CD)---for solving \eqref{eq:main}, enjoying a fast rate similar to methods such as SDCA, SAG, S2GD, SVRG, MISO, SAGA, mS2GD and QUARTZ. S2CD can be seen as a hybrid between S2GD and CD. In particular, the complexity of our method  is the sum of two terms: \[O(n\log (1/\epsilon))\] evaluations $\nabla f_i$ (that is, $\log(1/\epsilon)$ evaluations of the gradient of $f$) and \[O(\hat{\kappa} \log(1/\epsilon))\] evaluations of $\nabla_j f_i$ for randomly chosen functions $f_i$ and randomly chosen coordinates $j$, where $\hat{\kappa}$ is a new condition number which is larger than  the condition number $\kappa$ appearing in the aforementioned methods. However, note that $\hat{\kappa}$ enters the complexity only in the term involving the cost of the evaluation of a partial derivative $\nabla_j f_i$, which can be substantially smaller than the evaluation cost of  $\nabla f_i$. Hence, our complexity result can be both  better or worse than previous results, depending on whether the increase of the condition number can or can not be compensated by the lower cost of the stochastic steps based on the evaluation of partial derivatives.

\paragraph{Outline.} The paper is organized as follows. In Section~\ref{sec:algo} we describe the S2CD algorithm and in Section~\ref{sec:complexity} we state a key lemma and our main complexity result. The proof of the lemma is provided in Section~\ref{subsec:key_lemma} and the proof of the main result in Section~\ref{sec:proofs}.

% Similar work can be found in~\cite{orbcdvr}, where the authors combined the coordinate descent method with prox-SVRG algorithm~\cite{proxsvrg}. Although their algorithm appears more general for the proximal part, it is not trivial to optimally exploit the properties of this scheme and our algorithm promises a better convergence rate, benefiting from nonuniform sampling and the sparsity of the problem itself.

\section{S2CD Algorithm}\label{sec:algo}

In this section we describe the Semi-Stochastic Coordinate Descent method (Algorithm~\ref{alg:S2CD}).

\begin{algorithm}[h]
\begin{algorithmic}
\State \textbf{parameters:}  $m$ (max \# of stochastic steps per epoch);  $\stepsize>0$ (stepsize parameter);  $x_0\in \R^d$ (starting point)
\For {$\oidx = 0, 1, 2, \dots$}
	\State  Compute and store $ \nabla f(\ovar_{\oidx}) = \tfrac{1}{n}\sum_i \nabla f_i(\ovar_{\oidx})$
	\State Initialize the inner loop: $\ivar_{\oidx, 0} \gets \ovar_{\oidx}$
	\State Let $\iidx_{\oidx} = T \in \{1,2,\dots,m\}$ with probability $\left(1 - \mu \stepsize\right)^{m-T} / \beta$ 
	\For {$\iidx = 0$ to $\iidx_{\oidx}-1$}
                \State Pick coordinate $\csample \in \{ 1, 2, \dots, d \}$ with probability $p_\csample$
		\State Pick function index $\rsample$ from the set $\{i\;:\;L_{ij}>0\}$ with probability  $q_{\rsample\csample}$
		\State Update the $j^{th}$ coordinate: $ \ivar_{\oidx, \iidx+1} \gets \ivar_{\oidx,\iidx} - \stepsize p_{\csample}^ {-1}  \big( \nabla_{\csample} f(\ovar_{\oidx}) + \frac{1}{n \weightinG_{\rsample \csample}} \left( \nabla_{\csample} f_{\rsample}(\ivar_{\oidx, \iidx}) - \nabla_{\csample} f_{\rsample}(\ovar_{\oidx}) \right) \big) e_{\csample}$
	\EndFor
	\State Reset the starting point: $\ovar_{\oidx+1} \gets \ivar_{\oidx, \iidx_{\oidx}}$
\EndFor
\end{algorithmic}

\caption{Semi-Stochastic Coordinate Descent (S2CD)}
\label{alg:S2CD}
\end{algorithm}

The method has an outer loop (an ``epoch''), indexed by  counter $\oidx$, and an inner loop, indexed by $\iidx$. At the beginning of epoch $k$, we compute and store the gradient of  $f$ at $\ovar_\oidx$.  Subsequently, S2CD enters the inner loop in which 
a sequence of vectors $y_{k,t}$ for $t=0,1\dots,t_k$ 
is computed in a stochastic way, starting from $y_{k,0}=\ovar_{\oidx}$. The number $t_{\oidx}$ of stochastic steps in the inner loop is  random, following a geometric law: 
\[\Prob(t_k=T)= \frac{(1-
\mu \stepsize )^{m-T}}{\beta}, 
\qquad T\in\{1,\dots,m\},\]
where
\[\beta \eqdef \sum_{t = 1}^m (1 - \mu \stepsize )^{m-t}.\] 
In each step of the inner loop, we seek to compute $y_{k,t+1}$, given $y_{k,t}$. In order to do so, we sample coordinate $j$ with probability $p_j$ and subsequently\footnote{In S2CD, as presented, coordinates $\csample$ is selected first, and then function $\rsample$ is selected, according to a distribution conditioned on the choice of $\csample$. However,  one could equivalently sample $(\rsample, \csample)$ with joint probability $p_{\rsample \csample}$. We opted for the sequential sampling for clarity of presentation purposes.} sample $i$ with probability $q_{ij}$, where the probabilities are given by
\begin{equation}\label{eq:sjs7tbjd}
\omega_\rsample \eqdef | \{ j : L_{\rsample \csample} \neq 0 \} |,\quad \weightofnorm_\csample \eqdef \sum_{\rsample = 1}^ n \omega_\rsample L_{\rsample \csample}, 
\quad
p_{\csample} \eqdef \weightofnorm_\csample / \sum_{\csample = 1}^d \weightofnorm_\csample, \quad  q_{\rsample \csample} \eqdef \frac{\omega_\rsample L_{\rsample \csample}}{\weightofnorm_j}, \quad  p_{\rsample \csample} \eqdef p_\csample q_{\rsample \csample}.
\end{equation}
Note that $L_{\rsample \csample}=0$ means that function $f_\rsample$ does not depend on the $\csample^{th}$ coordinate of $x$. Hence, $\omega_\rsample$ is the number of coordinates  function $f_\rsample$ depends on -- a measure of sparsity of the data\footnote{The quantity $\omega\eqdef \max_i \omega_i$ (degree of partial separability of $f$) was used in the analysis of a large class of randomized parallel  coordinate descent methods in \cite{pcdm}. The more informative quantities $\{\omega_i\}$ appear in the analysis of parallel/distributed/mini-batch coordinate descent methods \cite{hydra, approx, hydra2}.}. It can be shown that $f$ has a $1$-Lipschitz gradient with respect to the weighted Euclidean norm with weights $\{v_j\}$ (\cite[Theorem 1]{approx}). Hence, we sample coordinate $j$ proportionally to this weight $v_j$. Note that $p_{\rsample \csample}$ is the joint probability of choosing the pair $(\rsample, \csample)$.

Having sampled coordinate $j$ and function index $i$, we compute two partial derivatives: $\nabla_\csample f_\rsample(\ovar_\oidx)$ and $\nabla_\csample f_\rsample(\ivar_{\oidx, \iidx})$ (we compressed the notation here by writing $\nabla_j f_i(x)$ instead of $\< \nabla f_i(x),e_j >$), and combine these with the pre-computed value $\nabla_j f(\ovar_{\oidx})$ to form an update of the form 
\begin{equation}\label{eq:87gsb8s9} y_{k,t+1} \leftarrow y_{k,t} - \stepsize p_j^{-1} G^{ij}_{kt} e_j = y_{k,t} - h g_{kt}^{ij},\end{equation}
where \begin{equation}\label{eq:g_kl}
g_{kt}^{ij}\eqdef p_j^{-1} G_{kt}^{ij}e_j
\end{equation} 
and
\begin{equation}\label{eq:0j9j0s9s}G^{ij}_{kt} \eqdef  \nabla_{\csample} f(\ovar_{\oidx}) + \frac{1}{n \weightinG_{\rsample \csample}} \left( \nabla_{\csample} f_{\rsample}(\ivar_{\oidx, \iidx}) - \nabla_{\csample} f_{\rsample}(\ovar_{\oidx}) \right).\end{equation}  
Note that only a single coordinate of $y_{k,t}$ is updated at each iteration. 

In the entire text (with the exception of the statement of Theorem~\ref{thm:S2CD} and a part of Section~\ref{eq:s98h9s8h}, where $\Exp$ denotes the total expectation) we will assume that all expectations are conditional on the entire history of the random variables generated up to the point when $y_{k,t}$ was computed. With this convention, it is possible to think that there are only two random variables: $j$ and $i$.  By $\Exp$ we then mean the expectation with respect to both of these random variables, and by $\Exp_i$ we mean expectation with respect to $i$ (that is, conditional on $j$). With this convention,  we can write  
\begin{eqnarray}\label{EGij}\Exp_i \left[G^{ij}_{kt} \right] 
& = & \sum_{i=1}^n q_{ij} G^{ij}_{kt} \notag\\
& \overset{\eqref{eq:0j9j0s9s}}{=} & \nabla_{\csample} f(\ovar_{\oidx}) + \frac{1}{n} \sum_{i=1}^n \left( \nabla_{\csample} f_{\rsample}(\ivar_{\oidx, \iidx}) - \nabla_{\csample} f_{\rsample}(\ovar_{\oidx}) \right) \;\;\overset{\eqref{eq:main} }{=}\;\; \nabla_j f(y_{k,t}),\label{eq:s98j6dd}
\end{eqnarray}
which means that conditioned on $j$, $G^{ij}_{kt}$ is an unbiased estimate of the  $j^{th}$ partial derivative of $f$ at $y_{k,t}$. An equally easy calculation reveals that  the random vector
$g_{kl}^{ij}$ is an unbiased estimate of the  gradient of $f$ at $y_{k,t}$:
\begin{eqnarray*}
\Exp \left[ g^{ij}_{kl} \right] &\overset{\eqref{eq:g_kl}}{=}& \Exp \left[p_j^{-1} G_{kt}^{ij}e_j \right] = \Exp \left[\Exp_i \left[ p_j^{-1} G^{ij}_{kt} e_j \right] \right] \\
&=& \Exp \left[p_{j}^{-1} e_j \Exp_i \left[G^{ij}_{kt}\right] \right]
\;\; \overset{\eqref{eq:s98j6dd}}{=} \;\;   \Exp \left [p_{j}^{-1} e_j \nabla_j f(y_{k,t}) \right] \;\;  = \;\; \nabla f(y_{k,t}).\end{eqnarray*}

Hence, the update step performed by S2CD is a stochastic gradient step of fixed stepsize $h$.  

Before we describe our main complexity result in the next section, let us briefly comment on a few special cases of S2CD:

\begin{itemize}
\item If  $n = 1$ (this can be always achieved simply by grouping all functions in the average into a single function), S2CD  reduces to a stochastic CD algorithm with importance sampling\footnote{A   parallel CD method  in which every subset of coordinates can be assigned a different probability of being chosen/updated was analyzed in \cite{nsync}.}, as studied in \cite{nesterov,rtsimple, quartz}, but written with many redundant computations. Indeed, the method in this case  does not require the $x_k$ iterates, nor does it need to compute the gradient of $f$, and instead takes on the form: \[y_{0,t+1}\leftarrow y_{0,t} - \stepsize p_j^{-1} \nabla_j f(y_{0,t})e_j,\] where $p_j=L_{1j}/\sum_j {L_{1j}}$.

\item It is possible to extend the S2CD algorithm and results to the case when coordinates are replaced by (nonoverlapping) blocks of coordinates, as in \cite{rtsimple} --- we did not do it here for the sake of keeping the notation simple. In such a setting, we would obtain semi-stochastic {\em block} coordinate descent. In the special case with {\em all variables forming a single block}, the algorithm reduces to the S2GD method described in \cite{s2gd}, but with nonuniform probabilities for the choice of  $i$ --- proportional to the Lipschitz constants of the gradient of the functions $f_i$ (this is also studied in \cite{proxsvrg}). As in \cite{proxsvrg}, the complexity result then depends on the average of the Lipschitz constants.
\end{itemize}

Note that the algorithm, as presented, assumes the knowledge of $\mu$. We have done this for simplicity of exposition: the method works also if $\mu$ is replaced by some lower bound in the method, which can be set to 0 (see \cite{s2gd}). The change to the complexity results will be only minor and all our conclusions hold. Likewise, it is possible to give an $O(1/\epsilon)$ complexity result in the non-strongly convex case $f$ using standard regularization arguments (e.g., see \cite{s2gd}).

%\iffalse
%Note that in Algorithm~\ref{S2CD-d233}, the probability that $(\rsample,\csample)$ is selected equals to $p_{\rsample,\csample}$, which is defined in~\eqref{eqn-pijqij}.
%Thus, Algorithm~\ref{S2CD-d233} has its equivalent Algorithm~\ref{S2CD-d23}, for which we next provide a theoretical anaysis.
%
%\begin{algorithm}[h]
%\begin{algorithmic}
%\State \textbf{parameters:} $m$ = max \# of stochastic steps per epoch, $\stepsize>0$
%\For {$\oidx = 0, 1, 2, \dots$}
%	\State $\grad_{\oidx} \gets  \nabla f(\ovar_{\oidx})$
%	\State $\ivar_{\oidx, 0} \gets \ovar_{\oidx}$
%	\State Let $\iidx_{\oidx} \gets t$ with probability $\left(1 - \stepsize\mu\right)^{m-t-1} / \beta $ for $t = 0, 2, \dots, m-1$
%	\For {$\iidx = 0$ to $\iidx_{\oidx}-1$}
%                \State Pick $\csample \in \{ 1, 2, \dots, d \}$, according to $p_\csample$;
%		\State Pick $\rsample \in Q_{\csample}$ at random, according to $q_{\rsample,\csample}$;
%		\State $ \ivar_{\oidx, \iidx+1} \gets \ivar_{\oidx,\iidx} - \stepsize p_{\csample}^ {-1} e_{\csample} \big[ \nabla_{\csample} f(\ovar_{\oidx}) + \frac{1}{n \weightinG_{\rsample, \csample}} \left( \nabla_{\csample} f_{\rsample}(\ivar_{\oidx, \iidx}) - \nabla_{\csample} f_{\rsample}(\ovar_{\oidx}) \right) \big] $
%	\EndFor
%	\State $\ovar_{\oidx+1} \gets \ivar_{\oidx, \iidx_{\oidx}}$
%\EndFor
%\end{algorithmic}
%
%\caption{Coordinate Semi-Stochastic Gradient Descent (S2GD)}
%\label{S2CD-d23}
%\end{algorithm}
%\fi

\section{Complexity Result}\label{sec:complexity}

In this section, we state and describe our complexity result; the proof is provided in Section~\ref{sec:proofs}. 

An important step in our analysis is  proving a  good upper bound on the variance of the (unbiased) estimator $g_{kt}^{ij} = p_j^{-1}G_{kt}^{ij}e_j$ of $\nabla f(y_{k,t})$, one that we can ``believe'' would  diminish to zero as the algorithm progresses.  This is important for several reasons. First, as the method approaches the optimum, we wish  $g_{kt}^{ij}$ to be progressively closer to the true gradient, which in turn will be close to zero. Indeed, if this was the case, then   S2CD behaves like gradient descent with fixed stepsize $h$ close to optimum. In particular, this would indicate that using fixed stepsizes makes sense. 

%Note that standard stochastic gradient descent (SGD) methods use the estimate $\nabla f_i(x)$ instead, where $x$ is the current point and $i$ is chosen uniformly at random. While this estimate is unbiased, there is no reason for these stochastic gradients to converge to zero. Even at optimality, such a scheme would drive the iterative process away. This is the reason for the need to decrease the stepsizes in standard SGD; the noise/variance  inherent in this estimate is tamed by decreasing the stepsizes.

%The variance of $g_{kl}^{ij}$ is given by \[  \Exp_{ij} \left[ \left\|g_{kl}^{ij} - \nabla f(y_{k,t})\right\|^2 \right]  = \Exp_{ij}\left[\left\| g_{kl}^{ij}\right\|^2\right] - \left\| \nabla f(y_{k,t}) \right\|^2 \]

In light of the above discussion, the following lemma plays a key role in our analysis:

\begin{lemma}\label{lem:main} The iterates of the S2CD algorithm satisfy 
\begin{equation}\label{eq:iuhs98s} \Exp \left[ \left\| g_{kt}^{ij}\right\|^2 \right] \leq 4 \hL \left( f(\ivar_{\oidx, \iidx}) - f(\ovar_*)\right) + 4 \hL \left( f(\ovar_{\oidx}) - f(\ovar_*) \right), \end{equation}
where
\begin{equation}\label{eq:us886vs5} \hL := \frac{1}{n}\sum_{\csample = 1}^d \weightofnorm_\csample \overset{\eqref{eq:sjs7tbjd}}{=} \frac{1}{n} \sum_{j = 1}^d \sum_{i = 1}^n \omega_i L_{ij}.\end{equation} 
\end{lemma}

The proof of this lemma can be found in Section~\ref{subsec:key_lemma}.

Note that as $y_{k,t}\to x_*$ and $\ovar_{\oidx}\to x_*$, the bound \eqref{eq:iuhs98s}  decreases to zero. This  is the main feature of modern fast stochastic gradient methods: the squared norm of the stochastic gradient estimate progressively diminishes to zero, as the method progresses, in expectation. 
 Note that the standard SGD method does not have this property: indeed, there is no reason for $\Exp_i \|\nabla f_i(x)\|^2$ to be small even if $x=x_*$. 
 
We are now ready to state the main result of this paper.

\begin{theorem}[Complexity of S2CD]
\label{thm:S2CD}
If $0< \stepsize  < 1/ (2\hL)$, then for all $k\geq 0$ we have:
\footnote{It is possible to replace modify the argument slightly and replace the term $\hat{L}$ appearing in the {\em numerator} by $\hat{L}-\frac{\mu}{\max_{s} p_s}$. However, as this does not bring any significant improvements, we decided to present the result in this simplified form.} 
\begin{equation} \label{eq:main} \Exp[f(\ovar_{\oidx+1}) - f(\ovar_*)] \leq \left(\frac{(1 - \mu \stepsize )^m}{(1 - (1 - \mu \stepsize )^m) (1 - 2\hL \stepsize ) } + \frac{2\hL \stepsize }{1 - 2\hL \stepsize }\right) \Exp[f(\ovar_{\oidx}) - f(\ovar_*)]. \end{equation}
\end{theorem}

By analyzing the above result (one can follow the steps in \cite[Theorem 6]{s2gd}), we get the following useful corollary:
\begin{corollary}
\label{cor:result}
Fix the number of epochs $\oidx \geq 1$, error tolerance $\epsilon \in(0,1)$ and let $\Delta \eqdef \epsilon^{1/\oidx}$ and $\hat{\kappa}\eqdef \hL/\mu$.
If we run Algorithm~\ref{alg:S2CD} with stepsize $\stepsize$ and $m$ set as
\begin{equation}\label{eq:sjs8s} \stepsize = \frac{\Delta}{(4   + 2\Delta)\hL}, \qquad m \geq \left( \frac{4}{\Delta}  + 2\right) \log \left( \frac{2}{\Delta} + 2 \right) \hat{\kappa}, \end{equation}
then $ \Exp[f(\ovar_{\oidx}) - f(\ovar_{*})] \leq \epsilon(f(\ovar_{0}) - f(\ovar_{*}))$. In particular, for $k=\lceil \log(1/\epsilon)\rceil$ we have $\tfrac{1}{\Delta} \leq \exp(1)$, and we can pick
\begin{equation}\label{eq:jd98ydO}k = \lceil \log(1/\epsilon)\rceil, \qquad \stepsize = \frac{\Delta}{(4+2\Delta) \hL} \approx\frac{1}{(4 \exp(1) +2) \hL} \approx \frac{1}{12.87 \hL},\qquad   m\geq 26\hat{\kappa}.\end{equation}
\end{corollary}

%If we set $\oidx \eqdef \lceil \log (1 / \epsilon)\rceil$ in Corollary~\ref{cor:result}, we have $1/\Delta \leq e = \exp(1)$, and hence 
%$m^* \overset{\eqref{eq:sjs8s}}{\leq}   \hat{\kappa}(4e + 2) \log(2e + 2) \leq 26 \hat{\kappa}$. So, the following choice of parameters  of S2CD will lead to an $\epsilon$-solution (in expectation, in relative scale):

If we run S2CD with the parameters set as in \eqref{eq:jd98ydO}, then in each epoch  the gradient of $f$ is evaluated once (this is equivalent to $n$ evaluations of $\nabla f_i$), and the partial derivative of some function $f_i$ is evaluated  $2m\approx 52\hat{\kappa}=O(\hat{\kappa})$ times. If we let $C_{grad}$ be the average cost of evaluating the gradient $\nabla f_i$ and $C_{pd}$ be the average cost of evaluating the partial derivative $\nabla_j f_i$, then the total work of S2CD can be written as
\begin{equation}
\label{eq:complexityc}
 (n {\cal C}_{grad} + m {\cal C}_{pd}) k  \overset{\eqref{eq:jd98ydO}}{=}  \mathcal{O} \left( (n  \mathcal{C}_{grad} +  \hat{\kappa} \mathcal{C}_{pd}) \log \left(\frac{1}{\epsilon}\right) \right),
\end{equation}

The complexity results of methods such as S2GD/SVRG \cite{s2gd, svrg, proxsvrg} and SAG/SAGA \cite{sag, saga}---in a similar but not identical setup to ours (these papers assume $f_i$ to be $L_i$-smooth)---can be written in a similar form: 
\begin{equation}
\label{eq:complexity} 
\mathcal{O}\left( (n  \mathcal{C}_{grad} +  \kappa \mathcal{C}_{grad}) \log \left( \frac{1}{\epsilon}\right)\right),
\end{equation}
where $\kappa = L/\mu$ and either $L=L_{max}\eqdef \max_i L_i$ (\cite{sag, svrg,s2gd, saga}), or $L=L_{avg}\eqdef \tfrac{1}{n}\sum_i L_i$  (\cite{proxsvrg}). The difference between our result \eqref{eq:complexityc} and existing results \eqref{eq:complexity} is in the term  $\hat{\kappa}{\cal C}_{pd}$ -- previous results have $\kappa {\cal C}_{grad}$ in that place. This difference constitutes a trade-off:  while $\hat{\kappa}\geq \kappa$ (we comment on this below), we clearly have ${\cal C}_{pd}\leq {\cal C}_{grad}$. The comparison of the quantities $\kappa {\cal C}_{grad}$ and $\hat{\kappa}{\cal C}_{pd}$  is not straightforward and is problem dependent. 

Let us now comment how do the condition numbers  $\hat{\kappa}$ and $\kappa_{avg}=L_{avg}/\mu$ compare. It can be show that (see \cite{rtsimple}) \[L_i \leq \sum_{j=1}^d L_{ij}\] and, moreover, this inequality can be tight. Since $\omega_i\geq 1$ for all $i$, we have
\[\hat{\kappa} = \frac{\hL}{\mu} \overset{\eqref{eq:us886vs5}}{=} \frac{1}{\mu n}  \sum_{j = 1}^d \sum_{i = 1}^n \omega_i L_{ij} 
%= \frac{1}{\mu n}  \sum_{i = 1}^n  \sum_{j = 1}^d \omega_i L_{ij}
%\geq \frac{1}{\mu n}  \max_i  \sum_{j = 1}^d \omega_i L_{ij}
\geq \frac{1}{\mu n} \sum_{i=1}^n   \sum_{j=1}^d L_{ij}  \geq \frac{1}{\mu n} \sum_{i=1}^n L_i  = \frac{L_{avg}}{\mu} = \kappa_{avg}.\]

Finally, $\hat{\kappa}$  can be smaller or larger than $\kappa_{max}\eqdef L_{max}/\mu$.

\section{Proof of Lemma~\ref{lem:main} } \label{subsec:key_lemma}

We will prove the following stronger inequality:
\begin{align}\label{a-EGij2bounds2gd}
\Exp \left[\left\|g_{kt}^{ij}\right\|^2\right] 
 \leq 4 \hL \left( f(\ivar_{\oidx, \iidx}) - f(\ovar_*)\right) + 4 \left( \hL - \frac{\mu}{\max_s p_s} \right) \left( f(\ovar_{\oidx}) - f(\ovar_*) \right).
\end{align}
Lemma~\ref{lem:main} follows by dropping the negative term.

\paragraph{STEP 1.} We first break down the left hand side of \eqref{a-EGij2bounds2gd} into $d$ terms each of which we will bound separately. By first taking expectation conditioned on $j$ and then taking the full expectation, we can write:
\begin{eqnarray}
\Exp \left[\left\|g_{kt}^{ij}\right\|^2\right] &\overset{\eqref{eq:g_kl}}{=} & \Exp \left[\Exp_i \left[ \|p_j^{-1} G^{ij}_{kt} e_j\|^2 \right] \right] \notag \\
&  = & \Exp  \left[ p_j^{-1}  \Exp_i  \left[ \left(G^{ij}_{kt}\right)^2\right]  \right] \;\; =\;\; \sum_{s=1}^ d p_s^{-1} \Exp_i \left[ \left( G^{\rsample s}_{\oidx \iidx} \right)^2  \right]. \label{eq:98gsjs9t87}
\end{eqnarray}

\paragraph{STEP 2.} We now further break each of these $d$ terms into three pieces. That is, for each $j=1,\dots,d$ we have:

\begin{eqnarray}
\Exp_i \left[ \left( G^{\rsample, \csample}_{\oidx, \iidx} \right)^2 \right]
&
\overset{\eqref{eq:0j9j0s9s}}{=} & \Exp_i \left[ \left( \nabla_{\csample} f(\ovar_{\oidx}) + \frac{ \nabla_{\csample} f_{\rsample}(\ivar_{\oidx, \iidx}) - \nabla_{\csample} f_{\rsample}(\ovar_{\oidx})}{n \weightinG_{\rsample, \csample}}  + \frac{ \nabla_{\csample} f_{\rsample}(x_*) - \nabla_{\csample} f_{\rsample}(x_*) }{n \weightinG_{\rsample, \csample}} \right)^2 \right] \notag \\
& = &  \Exp_i \left[ \left( \frac{\nabla_{\csample} f_{\rsample}(\ivar_{\oidx, \iidx}) - \nabla_{\csample} f_{\rsample}(x_*)}{n \weightinG_{\rsample, \csample}}  + \nabla_{\csample} f(\ovar_{\oidx}) - \frac{\nabla_{\csample} f_{\rsample}(\ovar_{\oidx}) - \nabla_{\csample} f_{\rsample}(x_*)}{n \weightinG_{\rsample, \csample}} \right)^2\right] \notag \\
& \leq & 2  \Exp_i \left[ \left( \frac{ \nabla_{\csample} f_{\rsample}(\ivar_{\oidx, \iidx}) - \nabla_{\csample} f_{\rsample}(x_*)}{n \weightinG_{\rsample, \csample}} \right)^2 \right]  + 2  \Exp_i \left[ \left( \nabla_{\csample} f(\ovar_{\oidx}) - \frac{\nabla_{\csample} f_{\rsample}(\ovar_{\oidx}) - \nabla_{\csample} f_{\rsample}(x_*) }{n \weightinG_{\rsample, \csample}} \right)^2 \right] \notag \\
& = & 2 \Exp_i \left[ \left( \frac{\nabla_{\csample} f_{\rsample}(\ivar_{\oidx, \iidx}) - \nabla_{\csample} f_{\rsample}(x_*)}{n \weightinG_{\rsample, \csample}} \right)^2 \right] + 2\Exp_i \left[ \left( \frac{\nabla_{\csample} f_{\rsample}(\ovar_{\oidx}) - \nabla_{\csample} f_{\rsample}(x_*)}{n \weightinG_{\rsample, \csample}}  - (\nabla_{\csample} f(\ovar_{\oidx}) - \nabla_{\csample} f(x_*)) \right)^2  \right] \notag \\
& = & 2  \Exp_i \left[ \left( \frac{\nabla_{\csample} f_{\rsample}(\ivar_{\oidx, \iidx}) - \nabla_{\csample} f_{\rsample}(x_*) }{n \weightinG_{\rsample, \csample}} \right)^2 \right]  + 2  \Exp _i\left[ \left(\frac{\nabla_{\csample} f_{\rsample}(\ovar_{\oidx}) - \nabla_{\csample} f_{\rsample}(x_*) }{n \weightinG_{\rsample, \csample}} \right)^2  \right]\notag \\
&& \qquad - 2  (\nabla_j f(\ovar_{\oidx}) - \nabla_{j} f(x_*))^2. \label{eq-dzfdfd}
\end{eqnarray}

\paragraph{STEP 3.} In this step we bound the first two terms in the right hand side of inequality \eqref{eq-dzfdfd}. It will now be useful to introduce the following notation:
\begin{equation}
\label{eq:Qj}
Q_\csample \eqdef \{ i : L_{\rsample \csample} \neq 0 \}, \qquad j=1,\dots,d,
\end{equation}
and 
\[
1_{\rsample \csample}  \eqdef \begin{cases}  1 & \mathrm{~if~} L_{ij}\neq 0 \\
 0 & \mathrm{~otherwise}
 \end{cases} , \qquad \rsample =1,\dots, n, \quad j=1,\dots,d.
\]

Let us fist examine the first term in the right-hand side of~\eqref{eq-dzfdfd}. Using the coordinate co-coercivity lemma (Lemma~\ref{lemma:coerc}) with $y=x_*$, we obtain the inequality
\begin{align}\label{a-dfzeff}
\left( \nabla_{\csample} f_{\rsample}(\ovar) - \nabla_{\csample} f_{\rsample}(\ovar_*) \right)^2 \leq 2L_{ij} \left( f_{\rsample}(\ovar) 
- f_{\rsample}(\ovar_*) - \left\< \nabla f_i(\ovar_*),\ovar - \ovar_* \right> \right),
\end{align}
using which we get the bound:
\begin{eqnarray}
& & 2 \sum_{s = 1}^d p_s^{-1} \Exp_i \left[ \left( \frac{1}{n \weightinG_{\rsample, \csample}} \left( \nabla_{\csample} f_{\rsample}(\ivar_{\oidx, \iidx}) - \nabla_{\csample} f_{\rsample}(x_*) \right)\right)^2 \right] \notag \\
&= & 2 \sum_{s = 1}^d p_s^{-1} \sum_{\rsample \in Q_s} \frac{1}{n^2 \weightinG_{\rsample, s}} ( \nabla_{s} f_{\rsample}(\ivar_{\oidx, \iidx}) - \nabla_{s} f_{\rsample}(x_*) )^2 \notag \\
& \overset{\eqref{a-dfzeff}}{\leq} & 4 \sum_{s = 1}^d p_s^{-1}\sum_{\rsample \in Q_s} \frac{L_{\rsample s}}{n^2 \weightinG_{\rsample, s}} \left( f_{\rsample}(\ivar_{\oidx, \iidx}) 
- f_{\rsample}(\ovar_*) - \left\< \nabla f_i(\ovar_*), \ivar_{\oidx, \iidx} - \ovar_* \right> \right) \notag \\
& \overset{\eqref{eq:Qj}}{=} & 4 \sum_{\rsample = 1}^n \sum_{s = 1}^ d p_s^{-1} 1_{\rsample s} \frac{ \weightofnorm_s}{n^2 \omega_\rsample}\left( f_{\rsample}(\ivar_{\oidx, \iidx}) - f_{\rsample}(\ovar_*) - \left\< \nabla f_i(\ovar_*), \ivar_{\oidx, \iidx} - \ovar_* \right> \right).\label{eq:iud7dke}
\end{eqnarray}

Note that by~\eqref{eq:sjs7tbjd} and \eqref{eq:us886vs5}, we have that for all $s =1,2,\dots,d$,
$$ p_s^ {-1}\weightofnorm_{s}=n\bar L.$$
Continuing from \eqref{eq:iud7dke}, we can therefore further write
\begin{eqnarray}
& & 2 \sum_{s = 1}^d p_s^{-1} \Exp_i \left[ \left( \frac{1}{n \weightinG_{\rsample, \csample}} \left( \nabla_{\csample} f_{\rsample}(\ivar_{\oidx, \iidx}) - \nabla_{\csample} f_{\rsample}(x_*) \right) \right)^2  \right] \notag \\
&\leq & 4 \sum_{\rsample = 1}^n \sum_{s = 1}^ d 1_{\rsample s} \frac{\hL}{n \omega_\rsample} \left( f_{\rsample}(\ivar_{\oidx, \iidx}) - f_{\rsample}(\ovar_*) - \left\< \nabla f_i(\ovar_*),\ivar_{\oidx, \iidx} - \ovar_* \right> \right) \notag \\
&= & \frac{4 \hL}{n}\sum_{i = 1}^n \left( f_{\rsample}(\ivar_{\oidx, \iidx}) - f_{\rsample}(\ovar_*) - \left\< \nabla f_i(\ovar_*),\ivar_{\oidx, \iidx}- \ovar_* \right> \right) \notag \\
&= & 4 \hL ( f(\ivar_{\oidx, \iidx}) - f(\ovar_*) ).\label{eq-dfdfs1}
\end{eqnarray}

The same reasoning applies to the second term on the right-hand side of the inequality~\eqref{eq-dzfdfd} and we have:
\begin{align}\label{a-derq2}
& 2 \sum_{s = 1}^d p_s^{-1} \Exp_i \left[ \left( \frac{1}{n \weightinG_{\rsample, \csample}} \left( \nabla_{\csample} f_{\rsample}(\ovar_{\oidx}) - \nabla_{\csample} f_{\rsample}(x_*) \right) \right)^2 \right] \leq 4 \hL (f(\ovar_{\oidx}) - f(\ovar_*)).
\end{align}

\paragraph{STEP 4.}

Next we bound the third term on the right-hand side of  the inequality~\eqref{eq-dzfdfd}. First note that since $f$ is $\mu$-strongly convex (see \eqref{SVRGstrcvx}), for all $x \in \R^ d$ we have:
\begin{align}\label{a-xxsna}
\left< \nabla f(x), x - \ovar_* \right> \geq  f(x) - f(\ovar_*) + \frac{\mu}{2}\| x - \ovar_* \|^2.
\end{align}
We can now write:
\begin{eqnarray}
 2 \sum_{s = 1}^d p_s^{-1} (\nabla_{s} f(\ovar_{\oidx})-\nabla_{s} f(x_*))^2  
&\geq & \frac{2}{\max_s p_s} \sum_{j = 1}^d (\nabla_{\csample} f(\ovar_{\oidx}) - \nabla_{\csample} f(x_*))^2\notag  \\
&\overset{\eqref{a-xxsna}}{ \geq} & \frac{4 \mu }{\max_s p_s}(f(\ovar_{\oidx}) - f(x_*)).\label{eq-eeff}
\end{eqnarray}

\paragraph{STEP 5.} We conclude by combining~\eqref{eq:98gsjs9t87}, \eqref{eq-dzfdfd},~\eqref{eq-dfdfs1},~\eqref{a-derq2} and~\eqref{eq-eeff}.

\section{Proof of the Main Result}\label{sec:proofs}

In this section we provide the proof of our main result.  In order to present the proof in an organize fashion,  we first establish two technical lemmas.

\subsection{Coordinate co-coercivity}

It is a well known and widely used fact (see, e.g. \cite{nesterovIntro}) that for a continuously differentiable function  $\phi:\R^d\to \R$ and constant $L_\phi>0$, the following two conditions are equivalent:
\[\phi(x) \leq \phi(y) + \< \nabla \phi(y),x-y> + \frac{L_\phi}{2}\|x-y\|^2, \quad \forall x,y\in \R^d\]
and
\[\|\nabla \phi(x) - \nabla \phi(y)\|^2 \leq 2L_{\phi} (\phi(x)-\phi(y)-\<\nabla \phi(y),x-y>), \qquad \forall x,y\in \R^d.\]
The second condition is often referred to by the name co-coercivity.  Note that our assumption \eqref{eq:sjs7shd} on $f_i$ is similar to the first inequality. In our first lemma we establish a coordinate-based co-coercivity result which applies to functions $f_i$ satisfying  \eqref{eq:sjs7shd}. 

 \begin{lemma}[Coordinate co-coercivity] \label{lemma:coerc}For all $x,y\in \R^d$ and $i=1,\dots, n$, $j=1,\dots,d$, we have:
\begin{align}\label{a-dfzeff1}
\left(\nabla_{\csample} f_{\rsample}(\ovar) - \nabla_{\csample} f_{\rsample}(y) \right)^2 \leq 2L_{ij} \left( f_{\rsample}(\ovar) 
- f_{\rsample}(y) - \left\< \nabla f_i(y),\ovar - y \right> \right).
\end{align}
\end{lemma}
\begin{proof}
Fix any $i,j$ and $y\in \R^d$.
Consider the function $g_i:\R^ d\rightarrow \R$ defined by:
\begin{equation}\label{eq:9sh98hs}
g_i(x)\eqdef f_{\rsample}(\ovar) 
- f_{\rsample}(y) - \left\< \nabla f_i(y),\ovar - y \right>.
\end{equation}
Then since $f_i$ is convex, we know that $g_i(x)\leq 0$ for all $x$, with $g_i(y)=0$. Hence, $y$ minimizes $g_i$. We also know that for any $x\in \R^d$:
\begin{equation}\label{eq:09s98hs}
\nabla_j g_i(x)=\nabla_j f_i(x)-\nabla_j f_i(y).
\end{equation}
Since $f_i$ satisfies \eqref{eq:sjs7shd}, so does $g_i$, and hence for all $x\in \R^d$ and $h\in \R$, we have 
\[g_i(x+he_j) \leq g_i(x) + \<\nabla g_i(x) , he_j  > + \frac{L_{ij}}{2}h^2.\]
Minimizing both sides in $h$, we obtain 
$$
g_i(x)-g_i(y)\geq \frac{1}{2L_{ij}} (\nabla_j g_i(x))^2,
$$ which together with \eqref{eq:9sh98hs} yields the result.
\end{proof}

\subsection{Recursion}

We now proceed to the  final lemma, establishing a key recursion which ultimately yields the proof of the main theorem, which we present in Section~\ref{eq:s98h9s8h}.

\begin{lemma}[Recursion]\label{l-gah0} The iterates of S2CD satisfy the following recursion:
\begin{equation}
\begin{split}
\frac{1}{2} \Exp \left[ \| \ivar_{\oidx, \iidx + 1} - \ovar_* \|^2 \right] &+ h(1-2h \hat L)(f(\ivar_{\oidx, \iidx}) - f(\ovar_*)) \\
&\leq (1-h\mu)\frac{1}{2}\|\ivar_{\oidx, \iidx} - \ovar_*\|^2+2h^2  \hat L (f(\ovar_{\oidx}) - f(\ovar_*)).
\end{split}
\end{equation}

\end{lemma}
\begin{proof}
 \begin{eqnarray*}
\frac{1}{2} \Exp \left[ \| \ivar_{\oidx, \iidx+1} - \ovar_* \|^2 \right]  &\overset{\eqref{eq:87gsb8s9}}{=} &\frac{1}{2} \Exp \left[ \left\| \ivar_{\oidx, \iidx} - h p_{\csample}^{-1} G^{\rsample \csample}_{\oidx \iidx} e_{\csample} - \ovar_* \right\|^2 \right] \\
&= &\frac{1}{2} \| \ivar_{\oidx, \iidx} - \ovar_* \|^2 - \Exp \left[ \left< h p_{\csample}^{-1} G^{\rsample \csample}_{\oidx \iidx} e_{\csample}, \ivar_{\oidx, \iidx} - \ovar_* \right> \right] + \frac{1}{2} \Exp \left[ \left\| h p_{\csample}^{-1} G^{\rsample \csample}_{\oidx \iidx} e_{\csample} \right\|^2 \right] \\
&\overset{\eqref{EGij}}{=} &\frac{1}{2} \| \ivar_{\oidx, \iidx} - \ovar_* \|^2 - h \left< \nabla f(\ivar_{\oidx, \iidx}), \ivar_{\oidx, \iidx} - \ovar_* \right> + \frac{h^2}{2} \Exp \left[ \left\| g_{kt}^{ij} \right\|^2 \right] \\
&\overset{\eqref{a-xxsna}}{\leq} &\frac{1}{2} \|\ivar_{\oidx, \iidx} - \ovar_* \|^2 - h \left( f(\ivar_{\oidx, \iidx}) - f(\ovar_*) + \frac{\mu}{2} \left\| \ivar_{\oidx, \iidx} - \ovar_* \right\|^2 \right) + \frac{h^2}{2} \Exp \left[ \left\| g_{kt}^{ij} \right\|^2 \right] \\
&\overset{\eqref{eq:iuhs98s}}{\leq} &\frac{1}{2} \|\ivar_{\oidx, \iidx} - \ovar_*\|^2 - h \left( f(\ivar_{\oidx, \iidx}) - f(\ovar_*) + \frac{\mu}{2} \| \ivar_{\oidx, \iidx} - \ovar_* \|^2 \right) \\
&& \qquad + 2 h^2 \hL \left( f(\ivar_{\oidx, \iidx}) - f(\ovar_*)\right) + 2 h^2  \hL \left(f(\ovar_{\oidx}) - f(\ovar_*) \right)
\\
&= & (1 - \mu h) \frac{1}{2} \|\ivar_{\oidx, \iidx} - \ovar_*\|^2 - h (1 - 2h \hL)(f(\ivar_{\oidx, \iidx}) - f(\ovar_*)) \\
&& \qquad + 2 h^2 \hL (f(\ovar_{\oidx}) - f(\ovar_*)).
\end{eqnarray*}
\end{proof}

\subsection{Proof of Theorem~\ref{thm:S2CD}}\label{eq:s98h9s8h}

For simplicity, let us denote:
$$
\eta_{\oidx, \iidx} \eqdef \frac{1}{2} \Exp \left[ \| \ivar_{\oidx, \iidx} - \ovar_* \|^2\right], \qquad  \xi_{\oidx, \iidx} \eqdef \Exp\left[ f(\ivar_{\oidx, \iidx}) - f(\ovar_*)\right],
$$
where the expectation now is with respect to the entire history. Then by Lemma~\ref{l-gah0} we have the following $m$ inequalities:
\begin{align*}
\eta_{\oidx, m} + h(1 - 2 h \hL) \xi_{\oidx, m-1} &\leq (1 - \mu h) \eta_{\oidx, m-1} + 2 h^2 \hL \xi_{\oidx, 0}, \\
(1 - \mu h) \eta_{\oidx, m-1} + h(1 - 2 h \hL)(1 - \mu h) \xi_{\oidx, m-2} &\leq (1 - \mu h)^2 \eta_{\oidx, m-2} + 2 h^2  \hL (1 - \mu h)\xi_{\oidx, 0}, \\
&\,\,\,\vdots \\
(1 - \mu h)^t \eta_{\oidx, m-t} + h(1 - 2 h \hL)(1 - \mu h)^t\xi_{\oidx, m-t-1} &\leq (1 - \mu h)^{t+1} \eta_{\oidx, m-t-1} + 2 h^2 \hL (1 - \mu h)^t \xi_{\oidx, 0}, \\
&\,\,\,\vdots \\
(1 - \mu h)^{m-1} \eta_{\oidx, 1} + \gamma(1 - 2 h \hL)(1 - \mu h)^{m-1} \xi_{\oidx, 0} &\leq (1 - \mu h)^m \eta_{\oidx, 0} + 2 h^2  \hL(1 - \mu h)^{m-1} \xi_{\oidx, 0}.
\end{align*}
By summing up the above $m$ inequalities, we get:
$$
\eta_{\oidx, m} + \gamma(1 - 2 h \hL) \beta \xi_{\oidx, m-1-t} \leq (1 - \mu h)^m \eta_{\oidx, 0} + 2 h^2 \hL \beta \xi_{\oidx, 0},
$$
where 
$$ \beta = \sum_{t = 0}^{m-1}(1 - \mu h)^t. $$
It follows from the strong convexity assumption~\eqref{SVRGstrcvx} that
$$ f(\ovar_k) - f(x_*) \geq \frac{\mu}{2} \| \ovar_k- x_* \|^2, $$
that is,
$$ \xi_{\oidx,0}\geq \mu \eta_{\oidx,0}. $$
Therefore,
$$ h(1 - 2 h \hL) \xi_{\oidx+1, 0} \leq \left( \frac{(1 - \mu h)^m}{\beta \mu} + 2 h^2 \hL \right) \xi_{\oidx, 0} $$
Hence if $0<2 h \hL < 1$, then we obtain:
$$
\xi_{\oidx+1, 0} \leq \left( \frac{(1 - \mu h)^m}{(1 - (1 - \mu h)^m)(1 - 2 h \hL)} + \frac{2 h \hL }{1 - 2 h \hL} \right) \xi_{\oidx, 0},
$$
which finishes the proof.

% \pagebreak
\bibliography{notes}

\end{document}